\documentclass[12pt]{amsart}

\usepackage[T1]{fontenc}
\usepackage[utf8]{inputenc}
\usepackage{amssymb}

\usepackage{amscd,verbatim}
\usepackage{enumitem}

\usepackage{color}
\usepackage[colorlinks,linkcolor=blue,citecolor=blue,urlcolor=red]{hyperref}

\newcommand{\Pic}{\operatorname{Pic}}
\newcommand{\NS}{\operatorname{NS}}
\newcommand{\Hom}{\operatorname{Hom}}

\newcommand{\End}{\operatorname{End}}

\newcommand{\Spec}{\operatorname{Spec}}

\newcommand{\Id}{\operatorname{Id}}

\newcommand{\sA}{\mathcal{A}}
\newcommand{\sB}{\mathcal{B}}
\newcommand{\sC}{\mathcal{C}}

\newcommand{\sM}{\operatorname{\mathbf{Mot}}}

\renewcommand{\L}{\mathbb{L}}

\newcommand{\C}{\mathbf{C}}
\newcommand{\G}{\mathbb{G}}

\newcommand{\Q}{\mathbf{Q}}

\newcommand{\Z}{\mathbf{Z}}
\newcommand{\GL}{\mathbf{GL}}
\newcommand{\bC}{\mathbb{C}}
\newcommand{\bU}{\mathbb{U}}

\newcommand{\un}{\mathbf{1}}
\newcommand{\SU}{\operatorname{S\mathbb{U}}}
\renewcommand{\Vec}{\operatorname{\mathbf{Vec}}}
\newcommand{\Rep}{\operatorname{\mathbf{Rep}}}
\newcommand{\LCorr}{\operatorname{\mathbf{LCorr}}}
\newcommand{\LMot}{\operatorname{\mathbf{LMot}}}
\newcommand{\Ab}{\operatorname{\mathbf{Ab}}}
\newcommand{\Abs}{\operatorname{\mathbf{Abs}}}
\newcommand{\Aut}{\operatorname{\mathbf{Aut}}}

\newcommand{\Ker}{\operatorname{Ker}}
\newcommand{\Coker}{\operatorname{Coker}}
\newcommand{\IM}{\operatorname{Im}}
\newcommand{\op}{{\operatorname{op}}}
\newcommand{\alg}{{\operatorname{alg}}}
\newcommand{\num}{{\operatorname{num}}}

\renewcommand{\epsilon}{\varepsilon}
\renewcommand{\phi}{\varphi}
\newcommand{\eff}{{\operatorname{eff}}}

\newcommand{\inj}{\hookrightarrow}
\newcommand{\surj}{\rightarrow\!\!\!\!\!\rightarrow}

\newcommand{\iso}{\overset{\sim}{\longrightarrow}}
\newcommand{\by}{\xrightarrow}
\newcommand{\tto}{\dashrightarrow}

\newtheorem{Thm}{Theorem}
\newtheorem{thm}{Theorem}[section]
\newtheorem{prop}[thm]{Proposition}
\newtheorem{cor}[thm]{Corollary}
\newtheorem{lemma}[thm]{Lemma}
\theoremstyle{definition}
\newtheorem{defn}[thm]{Definition}
\theoremstyle{remark}
\newtheorem{rk}[thm]{Remark}
\newtheorem{rks}[thm]{Remarks}
\newtheorem{qn}[thm]{Question}

\newtheorem{Ex}{Example}

\newcounter{spec}
\newenvironment{thlist}{\begin{list}{\rm{(\roman{spec})}}%
{\usecounter{spec}\labelwidth=20pt\itemindent=0pt\labelsep=10pt}}%
{\end{list}}%
\numberwithin{equation}{section}
\begin{document}
\title{Chow-Lefschetz motives}
\author{Bruno Kahn}
\address{CNRS, Sorbonne Université and Université Paris Cité, IMJ-PRG\\ Case 247\\4 place
Jussieu\\75252 Paris Cedex 05\\France}
\email{bruno.kahn@imj-prg.fr}
\date{April 9, 2024}
\begin{abstract}
We develop Milne's theory of Lefschetz motives for general adequate equivalence relations and over a not necessarily algebraically closed base field. The corresponding categories turn out to enjoy all properties predicted by standard and less standard conjectures, in a stronger way: algebraic and numerical equivalences agree in this context. We also compute the Tannakian group associated to a Weil cohomology in a different and more conceptual way than Milne's case-by-case approach.
\end{abstract}
\keywords{Motives, abelian varieties, divisors}
\subjclass[2020]{14K05, 14C15, 14C20}
\maketitle

\hfill {\it To the memory of Jacob Murre, one of the kindest human beings I have known.}

\tableofcontents

\section*{Introduction}

In two fundamental papers \cite{milne,milne2}, Milne put together various ``Lefschetz groups'' which had been attached to abelian varieties in the literature, showing that they assemble to form a Tannakian group associated to a category $\LMot(k)$ of ``Lefschetz motives'' over an algebraically closed field $k$.\footnote{To avoid confusions with \emph{the} Lefschetz motive and Lefschetz decompositions, and also to honour Milne's invention, it might be better to change the name ``Lefschetz motives'' to ``Milne motives''. We shall refrain from it here, but hope this suggestion will be taken up by other mathematicians.} This rigid $\otimes$-category has remarkable properties: homological and numerical equivalences agree for any Weil cohomology, and the fixed points of the Lefschetz group on the cohomology of an abelian $k$-variety $A$ consist of the subring generated by divisor classes. One can specialise Lefschetz motives of CM abelian varieties from characteristic $0$ to characteristic $p$, which allowed Milne to prove his famous theorem: the Hodge conjecture for CM abelian varieties over $\C$ implies the Tate conjecture for abelian varieties over finite fields \cite[Th. 7.1]{milne}. André used Milne's method in  \cite{andre} to prove an unconditional result: all Tate cycles over an abelian variety in positive characteristic are ``motivated'' in his sense.  
Three questions arise from Milne's construction:

\begin{enumerate}
\item \label{q1} Does $\LMot(k)$ make sense for other adequate equivalence relations than homological and numerical equivalences? This question is raised in \cite[p. 671]{milne2}.
\item\label{q2} Can one define $\LMot(k)$ when $k$ is not algebraically closed?\footnote{In \cite[Introduction]{milne}, it is announced that $\LMot(k)$ is constructed over any field $k$, but \cite[1.4]{milne} refers to \cite[5.5]{milne2} which assumes $k$ algebraically closed (the $k$ of \cite{milne} is the $\Omega$ of \cite{milne2}).}
\item\label{q3} Can one give a direct proof of Milne's computation of the Tannakian group of his category in terms of Lefschetz groups, which follows in \cite[\S 3]{milne2} from a case-by-case analysis?
\end{enumerate}

The aim of this paper is to answer these questions in the affirmative.

The first two questions amount to asking whether sums of intersection products of divisor classes on abelian varieties are preserved under direct images by projections of the form $A\times B\to B$, where $A,B$ are abelian varieties. For $k$ algebraically closed and for homological equivalence, Milne proves this in \cite[Cor. 5.5]{milne2} as a consequence of his computation of Lefschetz groups. Here, we prove it in general. Namely, for an abelian variety $A$ over a field $k$, let  $CH^*(A)$ denote the Chow ring of $A$, with rational coefficients, and  let $L^*(A)$ be the subring of $CH^*(A)$ generated by $CH^1(A)$. 

\begin{Thm}\label{t2} Let $f:A\to B$ be a homomorphism of abelian varieties. Then $f_* L^*(A)\subseteq L^*(B)$.
\end{Thm}

The proof of Theorem \ref{t2} is completely different from the one of \cite{milne2}, and relies on Beauville's computations in \cite{beauville1}, which ultimately depend on the Fourier transform on abelian varieties. More precisely, it is a consequence of the following result. Let $P_*(A)$ be the subring of $CH_*(A)$, for the Pontrjagin product, generated by $CH_0(A)(k)$ and $CH_1^{(0)}(A)$, where $CH_0(A)(k)$ is the subgroup of $CH_0(A)$ generated by the $0$-cycles $[a]$ for $a\in A(k)$ and $CH_1^{(0)}(A)$ is Beauville's  eigenspace of weight $0$ in $CH_1(A)$ (\cite{beauville2}, see Notation below). 

\begin{Thm}\label{p1} $L^*(A) = P_*(A)$.
\end{Thm}

As Beauville pointed out, the version of this theorem for numerical equivalence was proven previously by Schoen \cite[Prop. 1.4]{schoen}, using a different method. See also Polishchuk \cite[Th. 2.1]{pol}.

Using Theorem \ref{t2}, we associate in Definition \ref{d1} a rigid $\otimes$-category of Lefschetz motives $\LMot_\sim(k)$ to any adequate equivalence relation $\sim$ on algebraic cycles. These categories turn out to have rather wonderful properties:

\begin{Thm}\label{t3} a) Let $\alg$ (resp. $\num$) denote algebraic (resp. numerical) equivalence.  Then the projection
\[\LMot_\alg(k)\to \LMot_\num(k)\]
is an isomorphism of semi-simple abelian categories. In particular, the analogue of Voevodsky's smash-nilpotence conjecture \cite{voe} holds, and so does \emph{a fortiori} the analogue of the standard conjectures $D$ and $C$.\\
b) The Künneth decompositions given by a) provide $\LMot_\num(k)$ with a canonical weight grading (Definition \ref{d5.1}). After changing the commutativity constraint as usual, $\LMot_\num(k)$ becomes Tannakian and any Weil cohomology yields a fibre functor.
\end{Thm}

This theorem is in stark contrast with the case of usual motives, where algebraic and homological equivalences do \emph{not} agree, by a famous theorem of Griffiths \cite{griffiths}. A more recent and relevant example, due to Schoen, is the cube of the Fermat elliptic curve $E$ in characteristic $0$ \cite[(0.2) and (14.1)]{schoen2}. As a consequence, $\dim_\Q\Coker(L^*(E^3)\to CH^*(E^3))=+\infty$; this is remarkable, since this map becomes an isomorphism modulo numerical equivalence by \cite{imai} (see also \cite{ff-ellcurves}).

In case $\sim$ is rational equivalence, we simply write $\LMot(k)$ for $\LMot_\sim(k)$. Then:

\begin{Thm}\label{t4} The analogue of Beauville's conjectures \cite{beauville2} and of Murre's conjectures \cite{murre} holds in $\LMot(k)$, which is a Kimura-O'Sullivan category.\footnote{Recall that a $\Q$-linear $\otimes$-category $\sA$ is \emph{Kimura-O'Sullivan} if any object $A\in \sA$ may be written as a direct sum $A_+\oplus A_-$ where $A_+$ (resp. $A_-$) is killed by some exterior (resp. symmetric) power: this is conjecturally the case for the category of Chow motives over any field, and is known for motives of abelian varieties.} The projection $\LMot(k)\to \LMot_\num(k)$ has a canonical symmetric mono\-id\-al section.
\end{Thm}

As can be expected, the proof of last statement uses O'Sullivan's lifting theorem \cite{os}.

Finally, let us discuss Question \ref{q3} above. Let $k_s$ be a separable closure of $k$; we write $\iota:\LMot_\num(k)\to \LMot_\num(k_s)$ for the ``extension of scalars'' functor (see discussion following Definition \ref{d1} in \S \ref{s4}).  Let $\omega:\LMot_\num(k_s)\to \Vec_K$ be a symmetric monoidal, $\Q$-linear functor, where $\Vec_K$ is the category of finite-dimensional vector spaces over a field $K$ of characteristic $0$: by Theorem \ref{t3}, we may choose $\omega$ coming from a Weil cohomology.  (Here we changed the commutativity constraint as in Theorem \ref{t3} b).) For any abelian $k$-variety $A$, define $\bU(A)$ to be the algebraic $K$-group whose $K$-points are given by
\begin{equation} \label{eq0.4}
\bU(A)(K) = \{(\phi,\lambda)\in C(A)^*\times K^*\mid \phi \rho_A(\phi) = \lambda 1_{\omega^1(A)}\}
\end{equation}
where $C(A)$ is the centraliser of $\End^0(A)\otimes_\Q K$ in $\End_K \omega^1(A)$ and $\rho_A$ is the involution of $C(A)$ induced by some polarisation $u$ of $A$ (it does not depend on the choice of $u$, see Lemma \ref{l6.1}). Here, $\omega^1(A) = \omega(Lh^1_\num(A))$, see Corollary \ref{c1}. Let $\nu_A:\bU(A)\to \G_m$ be the character given by $(\phi,\lambda)\mapsto \lambda$, and $\SU(A)=\Ker \nu_A$.

\begin{Thm}[\protect{\cite[Prop. 1.8]{milne}} when $k$ is algebraically closed]\label{t5}\

a) Let $G$ (resp. $G_s$) be the affine $K$-group representing $\Aut^{\otimes}(\omega\iota)$ (resp.  $\Aut^{\otimes}(\omega)$). Then there is  
an exact sequence
\[1\to G_s \to G\to \Gamma \to 1\]
where $\Gamma=Gal(k_s/k)$, and an exact sequence 
\[1\to \prod_{A\in S} \SU(A)\to G_s\by{\nu}\G_m\to 1\]
where $S$ is the set of isogeny classes of simple abelian $k_s$-varieties, and $\nu$ is induced by the $\nu_A$'s. (The corresponding groups $\SU(A)$ were computed in \cite[Table 2 p. 655]{milne2}.)

b) Let $A$ be an abelian $k$-variety, $\langle A\rangle$ the  Tannakian subcategory of $\LMot_\num(k)$ generated by the Lefschetz motive $Lh_\num(A)$ of $A$, and $G_A=\Aut^{\otimes}(H_{\langle A\rangle})$. Let $E/k$ be the smallest Galois extension such that $\End^0(A_E)=\End^0(A_{\bar k})$. Then there is an exact sequence
\[1\to \bU(A_E) \to G_A\to Gal(E/k) \to 1\]
\end{Thm}

\begin{Ex} Let $A$ be an elliptic curve. If $A$ has no complex multiplication, then $E=k$ and $\bU(A)=\GL_2$. If $A$ has complex multiplication by $F$ in characteristic $0$, then $\bU(A_E) = R_{F/\Q} \G_m$ and $E=kF$. If $A$ is ordinary in charactertistic $p$, then $E=k$, while if it is supersingular,  then $E=k F_0$ if $k_0$ is the field of constants of $E$ and $F_0$ is the smallest (at most quadratic) extension of $k_0$ such that $\End^0(A_{F_0})$ is a quaternion algebra.
\end{Ex}

It may be possible to generalise the present theory of Lefschetz motives to the case of abelian schemes over a base, in the style of Deninger-Murre \cite{den-murre}: this seems likely in view of their results and those of Künnemann \cite{ku} as well as Polishchuk \cite{pol}, and may have an interest in view of Ancona's work \cite{ancona}. I haven't attempted it, however, and leave this problem to the interested reader.

This paper was conceived in 2019, but straightening out the computation of the Tannakian group took two more years, especially in the non-algebraically closed case: this is done in Section \ref{s6} which is certainly the most technical of this paper and occupies almost half of it. I had actually planned to give a computation of Deligne's fundamental group \cite[\S 8]{deligne0}, but decided to give up for now in order not to hold this work forever.

The reader may consult \cite{stacks} for 1) an interpretation of the two-step process of \S \ref{s4} as a ``stackification'' and 2) a simpler and more conceptual replacement of the profinite construction in \S \ref{s5.1}.

\section{Notation}

Let $A$ be an abelian variety of dimension $g$ over a field $k$. All Chow groups are tensored with $\Q$.  If $x\in CH_0(A)_0$, we write
\[\log_*(1-x) = - \sum_{n=1}^g \frac{x^{*n}}{n} \]
where $*$ denotes the Pontrjagin product. We set for $p\ge 0$
\[\gamma^p(x) = \frac{x^p}{p!}, \quad \gamma_*^p(x) = \frac{x^{*p}}{p^!}.\]

Instead of Beauville's notation $CH^p_s(A)$ \cite{beauville2}, we shall use $CH^p_{(s)}(A)$ in order to avoid confusion between $p$ and $s$, as we also use the notation
\[CH_q^{(s)}(A) = CH^{g-q}_{(s)}(A).\]

Recall that
\[CH^p_{(s)}(A)= \{x\in CH^p(A) \mid [k]^* x = k^{2p-s} x\ \forall\ k\in \Z \}\]
where $[k]$ is multiplication by $k$ on $A$.

For any $a\in A(k)$ and any $x\in CH^1(A)$, we write
\[\phi_a(x) = a^*x - x\]
where $a^*$ means ``pull-back by $a$'',  and
\[L(d,a_1,\dots,a_p,q) = \gamma^q(d) \phi_{a_1}(d)\dots \phi_{a_p}(d),\quad  p,q\ge 0, a_i\in A(k)\]
for $d$ an ample symmetric class in $CH^1(A)$. (Recall that an element $x\in CH^*(A)$ is \emph{symmetric} (resp. \emph{antisymmetric}) if $\sigma^*x = x$ (resp. $\sigma^*x = -x$), where $\sigma=[-1]$.)

If $\sim$ is an adequate equivalence relation on algebraic cycles, we write $A^*_\sim(A)$ for algebraic cycles on $A$ modulo $\sim$, and similarly $L^*_\sim(A)$ (so $A^*_\sim(A)= CH^*(A)$ and $L^*_\sim(A)=L^*(A)$ if $\sim$ is rational equivalence). We write $\sM_\sim(k)=\sM_\sim$ for the category of pure motives modulo $\sim$  \cite{scholl}, and simply $\sM(k)=\sM$ if $\sim$ is rational equivalence. We shall need the following formula:
\begin{equation}\label{eqb-1}
CH^p_{(s)}(A)=\sM(\L^p,h^{2p-s}(A))
\end{equation}
where $\L$ is the Lefschetz motive and $h^i(A)$ is the direct summand of the motive $h(A)$ of $A$ defined by the canonical $i$-th Chow-K\"unneth projector of Deninger-Murre \cite[Th. 3.1]{den-murre}: \eqref{eqb-1} is clear since $[k]$ acts on $h^i(A)$ by $k^i$ (loc. cit.).

\section{Technical lemmas}

\begin{lemma}\label{l0} $L^*(A)$ is generated by the $L(d,a_1,\dots,a_p,q)$ (as a $\Q$-vector space).
\end{lemma}

\begin{proof} Any element of $CH^1(A)$ is a sum of a symmetric and an antisymmetric class.  Thus $L^*(A)$ is $\Q$-linearly generated by products of the form
\[d_1\dots d_r x_1\dots x_s\]
with the $d_i$ symmetric and the $x_i$ antisymmetric. Since any symmetric class is a difference of two ample symmetric classes, we may restrict to those products where all $d_i$ are ample. But we may write $d_1\dots d_r$ as a linear combination of elements of the form $\gamma^r(\sum_I d_i)$ where $I$ runs through the subsets of $\{1,\dots,r\}$. Since $\sum_I d_i$ is ample for any such $I$, we see that 
$L^*(A)$ is $\Q$-linearly generated by products of the form
\[\gamma^r(d) x_1\dots x_s\]
for $d$ symmetric ample and the $x_i$ antisymmetric. Recall that the antisymmetric classes constitute $\Pic^0(A)\otimes \Q=A^*(k)\otimes \Q$.

Given such a $d$,  the map $a\mapsto \phi_a(d)$ defines an isogeny $A\to A^*$, hence an isomorphism $A(k)\otimes \Q\iso A^*(k)\otimes \Q$. Therefore the $\phi_a(d)$ generate $\Pic^0(A)\otimes \Q$, whence the conclusion.
\end{proof}

\begin{lemma} \label{l1} $CH_1^{(0)}(A)$ is generated by the $c_d=\gamma^{g-1}(d)$, where $d$ runs through the ample symmetric divisor classes on $A$.
\end{lemma}

\begin{proof} Since $d\in CH^1_{(0)}(A)$, $c_d\in CH_1^{(0)}(A)$. On the other hand,  we have an isomorphism of Chow motives
\[ h^2(A)\otimes \L^{g-2}\iso h^{2g-2}(A)\]
(\cite{ku}, \cite[Th. 5.2 (iii)]{scholl}), whence an isomorphism (see \eqref{eqb-1})
\begin{multline*}
CH^1_{(0)}(A) = \sM(\L,h^2(A))\iso \sM(\L^{g-1},h^2(A)\otimes \L^{g-2})\\
\iso \sM(\L^{g-1},h^{2g-2}(A))= CH^{g-1}_{(0)}(A)
\end{multline*}
given by cup-product with $\gamma^{g-2}(d)$. Therefore $CH^{g-1}_{(0)}(A)$ is generated by elements of the form $\gamma^{g-2}(d)\cdot d'$, where $d'$ is another ample symmetric divisor class. But $\gamma^{g-2}(d)\cdot d'$ may be written as a $\Q$-linear combination of elements of the form $\gamma^{g-1}(d+ad')$ for $a$ an integer $\ge 0$, which concludes the proof.
\end{proof}

\begin{lemma}\label{l2} a) Let $f:A\to B$ be a homomorphism of abelian $k$-varieties. Then we have
\[ f_*x *f_* y=f_*(x*y)\]
for any $x,y\in CH_*(A)$.\\
b) Let $A$ be an abelian $k$-variety, and let $l/k$ be a finite extension. Write $A_l$ for $A\otimes_k l$ and $p:A_l\to A$ for the projection. Then we have
\[x * p_* y=p_*(p^*x*y)\]
for any $(x,y)\in CH_*(A)\times CH_*(A_l)$ (projection formula for the Pontrjagin product).
\end{lemma}

\begin{proof} a) We use the following fact: if $X,Y$ are $k$-varieties, write $\times: CH_*(X)\times CH_*(Y)\to CH_*(X\times_k Y)$ for the cross-product. Let $X\by{u} X'$, $Y\by{v} Y'$ be two proper morphisms of $k$-varieties. Then we have
\[(u\times u)_*(a\times b) = u_*a\times v_*b\]
for any $(a,b)\in CH_*(X)\times CH_*(Y)$. For $u=v=f$, this gives
\[f_*x *f_* y = \mu^A_*(f_*x\times f_*x) = \mu^A_*(f\times f)_*(x\times y) = f_* \mu^B_*(x\times y) = f_*(x*y)\]
where $\mu^A$ and $\mu^B$ are the multiplication maps of $A$ and $B$.

b) Let $\mu$ be the multiplication map of $A$, and let $\pi:\Spec l\to \Spec k$ be the projection. Then the multiplication map $\mu^l$ of $A_l$ is $\mu\times 1_{\Spec l} $ modulo the identification $\alpha:A_l \times_l A_l \iso A\times_k A\times_k \Spec l$, while $p=1_A\times \pi$. 
Write $\tilde p$ for the projection $(1_{A\times_k A}\times \pi)\circ \alpha:A_l\times_l A_l\to A \times_k A$. Then
\[ \tilde p_*(p^*x\times_l y) = x\times p_*y\]
hence
\[
x*p_*y = \mu_* \tilde p_*(p^*x\times_l y)=p_* \mu^l_*(p^*x\times_l y) = p_*(p_*x*y)
\]
as desired.
\end{proof}

\begin{lemma}\label{l3} Let $A$ be an abelian $k$-variety. Write $L^p_{(s)}(A)=L^p(A)\cap CH^p_{(s)}(A)$. Then the homomorphism
\[\rho:L^p_{(0)}(A)\to L^p_\num(A) \]
induced by $L^p(A)\to L^p_\num(A)$ is bijective for any $p\ge 0$.
\end{lemma}

\begin{proof} Since $CH^1_{(0)}(A)=L^1_{(0)}(A)$  generates $L^*_{(0)}(A)$ multiplicatively and since $CH^1_{(0)}(A)\to A^1_\num(A)$ is surjective, $\rho$ is surjective. On the other hand, O'Sullivan \cite[Th. 6.1.1]{os} has constructed a ring-theoretic section of the homomorphism $CH^*(A)\to A^*_\num(A)$ which  sends  $A^1_\num(A)$ into $CH^1_{(0)}(A)$, hence restricts to a ring-theoretic section $\sigma$ of $\rho$. But $CH^1_{(0)}(A)\to A^1_\num(A)$ is even an isomorphism, hence $\sigma$ is surjective.  
\end{proof}

\section{Proofs of Theorems \ref{t2} and \ref{p1}}

\begin{proof}[Proof of Theorem \ref{p1}] It suffices to show:
\begin{enumerate}
\item[(A)]  $CH_0(A)(k)\subset L^*(A)$ and $CH_1^{(0)}(A) \subset L^*(A)$.
\item[(B)] $L^*(A)$ is stable under $*$.
\item[(C)] $L^*(A)\subseteq P_*(A)$.
\end{enumerate}

For this, we use the computations of Beauville in \cite{beauville1}, namely: 
\begin{equation}\label{eqb2}
\gamma^p(d) = \nu_d \gamma_*^{g-p}(c_d), \quad 0\le p\le g
\end{equation}
with $\nu_d = h^0(d)$ \cite[p 249, Cor. 2]{beauville1}, and
\begin{equation}\label{eqb1}
L(d,a_1,\dots,a_p,q)= (-1)^p\gamma^{p+q}(d)*\log_*[a_1]*\dots *\log_*[a_p]
\end{equation}
\cite[p. 250, Prop. 6]{beauville1}.

(These  computations are made over $k=\C$, but they are valid over any base field.)
Putting \eqref{eqb2} and \eqref{eqb1} together, we get
\begin{equation}\label{eqb3}
L(d,a_1,\dots,a_p,q)= (-1)^p\nu_d \gamma_*^{g-p-q}(c_d)*\log_*[a_1]*\dots *\log_*[a_p].
\end{equation}

In (A), for the first statement it suffices to show that $[a]\in L^*(A)$ for any $a\in A(k)$. By \eqref{eqb2} for $p=g$, we have $\gamma^g(d)=\nu_d [0]$, hence
\[\gamma^g(a^*d)=\nu_d [a].\]

The second statement of (A) follows immediately from Lemma \ref{l1}.

In view of Lemma \ref{l0}, (B) follows immediately from \eqref{eqb3}. Similarly, for (C) it suffices by Lemma \ref{l0} to show that all $L(d,a_1,\dots,a_p,q)$ belong to $P_*(A)$, which follows again from \eqref{eqb3}. 
\end{proof}

\begin{proof}[Proof of Theorem \ref{t2}] In view of Theorem \ref{p1}, it suffices to show that $f_*P_*(A)\subseteq P_*(B)$. Clearly, $f_*CH_0(A)(k)\subseteq CH_0(B)(k)$ and $f_*CH_1^{(0)}(A)\allowbreak\subseteq CH_1^{(0)}(B)$ by \cite[Prop. 2 c)]{beauville2}. We conclude with Lemma \ref{l2} a).
\end{proof}

\section{Categories of Lefschetz motives; proofs of Theorems \ref{t3} and \ref{t4}}\label{s4}

\subsection{A crude category} As a consequence of Theorem \ref{t2}, we have:

\begin{cor}\label{c1} For any adequate equivalence relation $\sim$ on algebraic cycles, there exists a pseudo-abelian rigid $\Q$-linear $\otimes$-category $\LMot_\sim(k)_0$, provided with two faithful symmetric monoidal functors 
\begin{equation}\label{eqb0}
\Ab(k)^\op\by{Lh_\sim}\LMot_\sim(k)_0\to \sM_\sim(k)
\end{equation}
where $\Ab(k)$ is the category of abelian $k$-varieties and $k$-morphisms, such that, for any $A,B\in \Ab(k)$,
\[\LMot_\sim(k)_0(Lh(A),Lh(B))=L^{\dim A}_\sim(B\times A)\]
where $L^*_\sim(A)$ is the image of $L^*(A)$ in $CH^*_\sim(A)$. The Chow-Künneth decomposition of Deninger-Murre and the Lefschetz isomorphisms of Künnemann \cite{ku} hold in $\LMot_\sim(k)_0$ (notation: $Lh^i_\sim(A)$). In particular, we have $Lh^i(A)=S^i(Lh^1(A))$ for any abelian variety $A$ and any $i\ge 0$, as well as a canonical isomorphism
\begin{equation}\label{eq5.1}
Lh^1(A^*)^\vee\simeq Lh^1(A)\otimes \L^{-1}
\end{equation}
defined by the Poincaré bundle class $P_A\in CH^1(A^*\times A)$, where $A^*$ is the dual abelian variety to $A$ and $()^\vee$ denotes duality in $\LMot_\sim(k)_0$. Moreover, if $f\in \Hom^0(A,B)$, we have $Lh^1(f)^\vee=Lh^1(f^*)$ modulo \eqref{eq5.1}, where $f^*$ is the dual homomorphism to $f$.
\end{cor}

\begin{rk}\label{r4.1} Here $S^i$ denotes the $i$th symmetric power for the symmetric monoidal structure where we don't modify the commutativity constraint; it is transformed into an exterior power by any fibre functor.
\end{rk}

\begin{proof} By Theorem \ref{t2}, composition of correspondences respects the subgroups $L^*_\sim(B\times A)\subseteq CH^*_\sim(B\times A)$. Therefore, there exists a $\Q$-linear additive category $\LCorr_\sim(k)_0$ whose objects are of the form $\coprod_{i\in I} A_i$, with $I$ finite and $A_i$ an abelian $k$-variety, and morphisms are
\[\LCorr_\sim(k)_0(\coprod_{i\in I} A_i,\coprod_{j\in J} B_j)=\prod_{(i,j)\in I\times J}L^{\dim B_j}_\sim(A_i\times B_j). \]

The existence of the functor $Lh_\sim:\Ab(k)^\op\to \LCorr_\sim(k)_0$ amounts to saying that the transpose of the class $[\gamma_f]$ of the graph $\gamma_f$ a morphism $f:A\to B$ belongs to $L^{\dim B}_\sim(B\times A)$. Note that this condition is preserved under direct products of morphisms.

When $f$ is a homomorphism, this follows from Theorem \ref{t2} since $[\gamma_f]=(\gamma_f)_*([A])$. (See \cite[1.3]{milne} or \cite[Th. 5.10]{milne2} for a different proof.) In general, $f$ is the composition of a homomorphism and a translation, so it remains to handle the latter case. But if $a\in A(k)$, the translation defined by $a$ may be written as a composition
\[A\times \Spec k\by{1_A\times a}A\times A\by{\mu} A \]
where we identify $a$ with the corresponding morphism $\Spec k \to A$ and $\mu$ is as usual the multiplication of $A$. Since $\mu$ is a homomorphism and the graph $[a]$ of $a$ belongs to $L^*(A)$  (Theorem \ref{p1}), we are done.

As usual, we define $\LMot_\sim^\eff(k)_0$ as the Karoubian envelope of $\LCorr_\sim(k)_0$. As observed in \cite[p. 672]{milne2}, the Chow-K\"unneth projectors of $A$ belong to $L^*(A\times A)$ as well as the Lefschetz isomorphisms of \cite{ku} and their inverses, which justifies the second claim already in $\LMot_\sim^\eff(k)_0$. Note in particular that  $\L\in \LMot^\eff_\sim(k)_0$ by using an elliptic curve (which makes it unnecessary to involve projective spaces in its definition as in \cite{milne}), so that we can define  $\LMot_\sim(k)_0$ by $\otimes$-inverting $\L$. Its rigidity is checked on additive generators $Lh_\sim(A)$ in the usual way, using the diagonal of $A\times A$ to define unit and counit, and \eqref{eq5.1} is also checked as usual.
 \end{proof}

We shall need the following lemma for the proof of Theorem \ref{t5}.

\begin{lemma}\label{l5.1} Let $A\in \Ab(k)$ and $n_1,\dots,n_r$ be natural integers. Then the composition
\begin{multline*} Lh^{n_1}(A)\otimes \dots\otimes Lh^{n_r}(A)\inj Lh(A)\otimes \dots \otimes Lh(A)\simeq Lh(A^r)\\
\by{Lh(\Delta)} Lh(A) \surj Lh^{n_1+\dots + n_r}(A)
\end{multline*}
equals the natural morphism obtained from the isomorphisms $Lh^{n_i}(A)\simeq S^{n_i}(Lh^1(A))$ of Corollary \ref{c1}. Here $\Delta$ is the diagonal embedding $A\inj A^r$.
\end{lemma}

\begin{proof} This is clear, since the morphism $\bigoplus_r Lh^1(A)\simeq Lh^1(A^r)\by{Lh^1(\Delta)} Lh^1(A)$ is the sum map.
\end{proof}

\subsection{The correct construction} When $k$ is not separably closed, the category $\LMot_\sim(k)_0$ is not large enough: it is ``without Artin motives'' (all abelian $k$-varieties are geometrically connected) and, perhaps more importantly, it doesn't cover enough abelian varieties and does not satisfy ``descent''. As a related issue, let $l/k$ be a finite extension.  There is a transfer map $t_{l/k}:CH^*(A_l)\to CH^*(A)$, but clearly $t_{l/k}(L^*(A_l))\not\subset L^*(A)$ since $t_{l/k}CH_0(A)(l)\not \subset CH_0(A)(k)$ if $l\ne k$. As a consequence, the extension of scalars functor $\LMot_\sim(k)_0\to \LMot_\sim(l)_0$ does not have a (left or right) adjoint if $l\ne k$. For these reasons, we enlarge $\LMot_\sim(k)_0$ as follows.

\begin{defn}\label{d1} Write $\Abs(k)$ for the category of abelian schemes over étale $k$-schemes. Let $\sim$ be an adequate equivalence relation. The category $\LCorr_\sim(k)$ has
\begin{description}
\item[objects] those of $\Abs(k)$.
\item[morphisms] Let $A,B$ be two such abelian schemes. Write $A\otimes_k k_s=\coprod_i A_i$ and $B\otimes_k k_s= \coprod_j B_j$, where $k_s$ is a separable closure of $k$ and $A_i, B_j$ are abelian $k_s$-varieties.  Then
\begin{equation}\label{eq5.4}
\LCorr_\sim(k)(A,B) = \left(\bigoplus_{i,j} L_\sim^{\dim A_i}(B_j\times_{k_s} A_i)\right)^\Gamma
\end{equation}
where $\Gamma=Gal(k_s/k)$.
\item[composition of morphisms] induced by that in $\LCorr_\sim(k_s)_0=\LCorr_\sim(k_s)$.
\end{description}
It has a tensor structure given by the product (over $k$).\\
We write $\LMot_\sim(k)$ for the rigid pseudo-abelian $\otimes$-category obtained out of $\LCorr_\sim(k)$ by the usual Grothendieck procedure. If $\sim$ is rational equivalence, we simply write $\LCorr(k)$ and $\LMot(k)$. 
\end{defn}

Let $E/k$ be a separable extension. If $A$ is an object of $\LCorr_\sim(k)$, then $A_E\otimes_E k_s=A_{k_s}$, whence a canonical $\otimes$-functor $\LMot_\sim(k)\to \LMot_\sim(E)$ sending $Lh(A)$ to $Lh(A_E)$. In particular, there is a canonical $\otimes$-action of $\Gamma$ on $\LMot_\sim(k_s)$\footnote{For $A\by{p} \Spec k_s\in \Ab(k_s)$ and $\sigma\in \Gamma$, define $\sigma_*A$ by $\sigma\circ p$.} such that $\sigma M_{k_s} = M_{k_s}$ for any $M\in \LMot_\sim(k)$ and $\sigma\in \Gamma$; from Definition \ref{d1}, we have
\begin{equation}\label{eq4.1}
\LMot_\sim(k)(M,N)\iso \LMot_\sim(k_s)(M_{k_s},N_{k_s})^\Gamma
\end{equation}
for any $M,N\in \LMot_\sim(k)$.

Note that $\LMot_\sim(k)_0$ enjoys the same functoriality in $k$, but the analogue of \eqref{eq4.1} is false.

\begin{lemma}\label{l4.1} Let $E/k$ be a finite separable extension. Then the extension of scalars functor $i_E:\LMot_\sim(k)\to \LMot_\sim(E)$ has a (non monoidal) right adjoint $\lambda_E$, sending a motive $Lh_\sim(A)$ to $Lh_\sim(A_{(k)})$ where $A_{(k)}\in \Abs(k)$ is the naïve restriction of scalars of $A\in \Abs(E)$. If $E/k$ is Galois, one has a natural isomorphism
\begin{equation}\label{eq5.6}
 i_E\lambda_E M\iso \bigoplus_{g\in G} g_*M
\end{equation}
for any $M\in \LMot_\sim(E)$, where $G=Gal(E/k)$.
\end{lemma}

\begin{proof} It suffices to prove that the given recipe for $\lambda_E$ defines a right adjoint to the functor $\LCorr_\sim(k)\to \LCorr_\sim(E)$; this corresponds to a natural isomorphism
\[\LCorr_\sim(E)(A_E,B)\simeq \LCorr_\sim(k)(A,B_{(k)})\]
for $B\in \Abs(k)$, which follows from \eqref{eq5.4}. If $E/k$ is Galois, the counit morphism $i_E\lambda_E\to \Id$ applied to $g_*M$ for all $g\in G$ yields a morphism \eqref{eq5.6}: it is an isomorphism if $M=Lh(A)\otimes \L^n$ for $A\in \Abs(E)$ and $n\in\Z$ by the isomorphism $A_{(k)}\otimes_k E\simeq \coprod_{g\in G} g^*A$, hence in general.
\end{proof}

\subsection{Proof of Theorem \ref{t3}} Apply Lemma \ref{l3} over $k_s$ and take Galois invariants. It follows that, for $A,B\in \LCorr(k)$, the homomorphism 
\[\LCorr(k)(A,B)\to \LCorr_\num(k)(A,B)\] 
restricts to an isomorphism on $\LCorr_0(k)(A,B)$, where $\LCorr_i(k)(A,B)\allowbreak\subseteq \LCorr(k)(A,B)$ is the subspace of Beauville weight $i$ (with respect to $A\times B$). Since this subspace obviously vanishes for $i\ne 0$ modulo algebraic equivalence, $\LCorr_\alg§k)\to \LCorr_\num(k)$ is an equivalence, and this extends canonically to $\LMot$. The semi-simplicity claim is proven as in \cite{jannsen} (or \cite{akcorr}), by using a Weil cohomology. The other claims of a) hold because algebraically trivial cycles are smash-nilpotent \cite{voe,voi}, hence homologically trivial. Note that, here, this merely follows from the fact that $Lh^1(A)$ is odd-dimen\-sion\-al in the sense of Kimura \cite{kim} for any abelian variety $A$ (reduce to the case of a cycle in $L^1_1(A)$). b) is then immediate, by using the weight theory of pure Hodge structures in characteristic $0$ and of $l$-adic representations in positive characteristic.

\subsection{Proof of Theorem \ref{t4}} The last statement follows from Lemma \ref{l3}, which also shows that the image of the section $\LMot_\num(k)\to \LMot(k)$ is the subcategory $\LMot_0(k)$ with the same objects, and such that
\begin{equation}\label{eqb6}
\LMot_0(k)(Lh(A),Lh(B)) = L^{\dim A}_0(B\times A).
\end{equation}

For an abelian variety $A$, Beauville's conjectures in \cite{beauville2} predict that
\begin{thlist}
\item $CH^p_{(0)}(A)\inj A^p_\num(A)$;
\item $CH^p_{(s)}(A)=0$ for all $s<0$.
\end{thlist}

For the subgroups $L^p_{(s)}(A)$, (i) is true by Lemma \ref{l3}, and (ii) follows from Lemma \ref{l0} since $L(d,a_1,\dots,a_p,q)\in L^{p+q}_{(p)}(A)$ for any $(p,q)$. We record this as
\bigskip
\begin{thlist}
\item[(iii)] $L^p_{(0)}(A)\inj L^p_\num(A)$;
\item[(iv)] $L^p_{(s)}(A)=0$ for all $s<0$.
\end{thlist}

 The existence of  the Deninger-Murre Chow-Künneth projectors (Cor\-ollary \ref{c1}) is the analogue of \cite[I, Conj. (A)]{murre}. Reasoning as in \cite[I, 2.5.4]{murre}, the corresponding analogue to Murre's filtration then verifies $F^\nu L^j(A) =\bigoplus_{s\ge \nu} L^j_s(A)$ for $0\le \nu \le j$ and $F^{j+1} L^j(A) =0$. Since $\Ker(L^g(A\times A)\to L^g_\num(A\times A))$ is a nilideal by \cite{kim}, any other choice of Chow-Künneth projectors is conjugate to the canonical one under a self-correspondence $1+n$ with $n\sim_\num 0$, as in \cite[Lemma 5.4]{jannsen}. By (iii) and (iv) applied to $A\times A$, we have $n\in L^g_{(>0)}(A\times A)$, hence $n L^j_{(s)}(A)\subseteq \bigoplus_{t>s} L^j_{(t)}(A)$. This implies that the filtration $F^\nu L^*(A)$ does not depend on the choice of a Chow-K\"unneth decomposition, which is the analogue of Murre's Conjecture (C). Finally, we get the analogue of Murre's conjecture (B) (resp. (D)) from (iv) (resp. (iii)). 

\section{Proof of Theorem \ref{t5}}\label{s6}

\subsection{Centralisers in Tannakian categories}\label{s7.1} Let $\sA$ be a Tannakian category, with $Q=\End_\sA(\un)$ a field of characteristic $0$, and let $\omega:\sA\to \Vec_K$ be a fibre functor, where $K$ is an extension of $Q$. We write $H=\Aut^\otimes \omega$ for the corresponding Tannakian group (an affine $K$-group).

\begin{lemma}\label{l7.2} One has the isomorphism 
\[\sA(N,N')\otimes_Q K\iso \Hom_K(\omega(N),\omega(N'))^H\]
for any $N,N'\in \sA$.
\end{lemma}

\begin{proof}  Recall \cite[III.1]{saa} (see also \cite[5.3.1]{akos}) that one can define a $K$-linear Tannakian category $\sA_K$ and a $Q$-linear $\otimes$-functor $(-)_K:\sA\to \sA_K$ such that $\sA(A,B)\otimes_Q K\to \sA_K(A_K,B_K)$ is an isomorphism for any $A,B\in \sA$ and that $\omega$ extends canonically to a fibre functor on $\sA_K$. This reduces us to the case where $Q=K$.
The claim then follows from Tannakian theory, since $\sA$ gets identified with $\Rep_K(H)$. 
\end{proof}

If $A$ is a $K$-algebra and $X$ is a subset of $A$, we write $C_A(X)$ for its centraliser.

\begin{lemma}\label{l1a} Let $A,B$ be two $K$-algebras, and let $X\subset A$, $Y\subset B$. Then $C_{A\otimes_K B}(X\otimes Y) = C_A(X)\otimes_K C_B(Y)$.
\end{lemma}

\begin{proof} One inclusion is obvious. Let us show the other. Without loss of generality, we may assume that $1_A\in X$ and $1_B\in Y$.

Let $c\in C_{A\otimes_K B}(X\otimes Y)$. Choose a $K$-basis $(a_i)$ of $A$, and write
\[c=\sum_i a_i\otimes b_i,\quad b_i\in B.\]

Writing $c(1_X\otimes y) = (1_X\otimes y) c$ for any $y\in Y$, we find that $b_i\in C_B(Y)$ for any $i$. From the $b_i$'s, extract now a maximal $K$-free  system $\beta_j$ and rewrite $c$ in the form
\[c=\sum_j \alpha_j \otimes \beta_j.\]

By the same reasoning with $x\otimes 1_Y$ for any $x\in X$, we find that $\alpha_j\in C_A(X)$ for all $j$.
\end{proof}

For $M\in \sA$, write $C_\omega(M)=C(M)$ for the centraliser of $\omega(\End_\sA(M))$ (or $\omega(\End_\sA(M))\otimes_Q K$) in $\End \omega(M)$.

\begin{prop}\label{p1a} For an integer $r>0$, write $rM$ for $\bigoplus_{i=1}^r M$. The objects $C(M)$ have the following properties:
\begin{thlist}
\item $C(M )\iso C(rM)$ for the diagonal homomorphism $\End  \omega (M )\allowbreak\inj \End  \omega (rM)$.
\item If $\sA(M ,N)=\sA(N,M )=0$, $C(M )\times C(N) \iso C(M \oplus N)$ for the inclusion $\End  \omega (M )\times \End \omega (N) \inj\End \omega (M \oplus N)$.
\item $C(M \otimes N) \subset C(M )\otimes_K C(N)$ for the isomorphism $\End \omega (M \otimes N)\iso \End \omega (M )\otimes_K \End \omega (N)$.
\item If $L\in \sA$ is invertible, $C(M )\simeq C(M \otimes L)$.
\item $C(M ^\vee)\simeq C(M )^\op$, where $M ^\vee$ is the dual of $M$.
\item If $M$ is semi-simple, $C(M)$ is semi-simple and $\End_\sA(M)$ is the centraliser of $C(M)$. In particular, the centres of $C(M)$ and $\End_\sA(M)$ coincide.
\end{thlist}
\end{prop}

\begin{proof} (i) and (ii) are matrix exercices. (iii) follows from the obvious homomorphism
\begin{equation}\label{eq12}
\End_\sA(M)\otimes_Q \End_\sA(N)\to \End_\sA(M\otimes N)
\end{equation}
and Lemma \ref{l1a}. (iv) follows from (iii) by taking $N=L$, then $N=L^{-1}$. (v) follows from the  compatible isomorphisms $\End(M ^\vee)\simeq \End(M )^\op$ and $\End \omega (M ^\vee)\simeq \End \omega (M )^\op$. Finally, (vi) follows from the double centraliser theorem  \cite[\S 14, no 5, th. 5 a)]{bbki}.
\end{proof}

Let $L\in \sA$ be an invertible object. In the sequel, we assume that $M$ is ``weakly polarisable with respect to $L$'': this means that there exists  an isomorphism 
\begin{equation}\label{eq10}
u:M\iso L\otimes M^\vee,
\end{equation} 
where $M^\vee$ is the dual of $M$. 
Then $u$ gives rise to a \emph{Rosati anti-automorphism}
\[\rho_{M,L,u}:\End \omega(M)\to (\End \omega(M))^\op;  \quad f\mapsto \omega(u)^{-1}(1_{\omega(L)} \otimes f^\vee) \omega(u)\]
which respects $\omega(\End_\sA(M))$, hence also $C(M)$.

Note that we don't require any symmetry property of $u$: it is not needed. Indeed:

\begin{lemma}\label{l6.2} The restriction of $\rho_{M,L,u}$ to $C(M)$ does not depend on the choice of $u$, and is an involution. We write it simply $\rho_{M,L}$ (or $\rho_M$, if $L$ is clear from the context).
\end{lemma}

\begin{proof} Let $u':M\iso L\otimes M^\vee$ be another weak polarisation. Then $u=(1_L\otimes v^\vee)u'$ for some $v\in \Aut(M)$, so the first claim is obvious. For the second one, we compute, for $f\in  \End \omega(M)$:
\begin{multline*}
\rho^2(f) = \omega(u)^{-1}(1_{\omega(L)} \otimes \rho(f)^\vee) \omega(u)\\
= \omega(u)^{-1}(1_{\omega(L)} \otimes (\omega(u)^{-1}(1_{\omega(L)} \otimes f^\vee) \omega(u))^\vee) \omega(u)\\
= \omega(u)^{-1}(1_{\omega(L)} \otimes (\omega(u)^\vee(1_{\omega(L^\vee)} \otimes f) (\omega(u)^{-1})^\vee)) \omega(u)\\
= \omega(u^{-1}u^\vee) f \omega((u^{-1})^\vee u)
\end{multline*}
with an obvious abuse of notation for $u^{-1}u^\vee$. If $f\in C(M)$, this equals $f$.
\end{proof}

Write $\langle M\rangle\subset \sA$ for the (full) sub-Tannakian category generated by $M$ (i.e. the smallest full subcategory of $\sA$ which contains $M$ and is closed under subquotients, extensions, tensor products and duals), and $H_M$ for $\Aut^\otimes(\omega_{\mid \langle M\rangle})$. It is a closed algebraic subgroup of $\GL_{\omega(M)}$ and a quotient of $H$, by \cite[Prop. 2.21 (a)]{dm}.

\begin{lemma}\label{l6.1} We have $L\in \langle M\rangle$.
\end{lemma}

\begin{proof} Indeed, $\un$ is a direct summand of $M\otimes M^\vee$ since the composition
\[\un \by{\eta} M^\vee\otimes M\by{\sigma} M\otimes M^\vee \by{\epsilon} \un\]
is multiplication by $\chi(M)=\dim \omega(M)\ne 0$, hence $L$ is a direct summand of $M\otimes M$.
\end{proof}

\begin{lemma}\label{l6} Let
\[U(M) =\{(\phi,\lambda)\in C(M)^*\times K^*\mid \phi \rho_M(\phi) = \lambda 1_{\omega(M)}\}.\]
Then $U(M)$ is a subgroup of $C(M)^*\times K^*$, and the homomorphism
\[H_M(K)\to C(M)^*\times K^*:\quad g\mapsto (g_M,g_L)\]
lands into $U(M)$.
\end{lemma}

\begin{proof} The first statement is obvious since $\rho$ is an anti-automorphism. For the second one, let $g\in H_M(K)$. We compute:
\begin{multline*}
g_M \rho(g_M) = g_M \omega(u)^{-1}(1_{\omega(L)} \otimes g_M^\vee) \omega(u) \\
=  \omega(u)^{-1}g_{L\otimes M^\vee} (1_{\omega(L)} \otimes g_M^\vee) \omega(u)\\
=\omega(u)^{-1}g_L\otimes g_{M^\vee} (1_{\omega(L)} \otimes g_M^\vee) \omega(u)\\  
=\omega(u)^{-1}g_L\otimes g_{M^\vee}  g_M^\vee \omega(u) \\
=  \omega(u)^{-1}g_L\otimes 1_{M^\vee} \omega(u) = g_L1_{\omega(M)}
\end{multline*}
since $\End \omega(L) = K$.
\end{proof}

We get the conclusion of Lemma \ref{l6}, with the same proof, after extending scalars from $K$ to any commutative $K$-algebra. This defines two closed immersions of algebraic $K$-groups 
\[H_M\subseteq \bU(M)\subset \bC(M)^*\times \G_m\]
(recall that $L\in \langle M\rangle$ by Lemma \ref{l6.1}). 

\begin{prop}\label{p5.1} The isomorphisms of Proposition \ref{p1a} (i) and (ii) induce respective isomorphisms
\begin{align*}
\bU(M)&\iso \bU(rM), &r>0\\
\bU(M)\times_{\G_m} \bU(N)&\iso \bU(M\oplus N) &\text{if } \sA(M,N)=\sA(N,M)=0
\end{align*}
where the fibre product is with respect to the second projections (cf. \cite[Def. 4.6]{milne2}). Here, we use the isomorphisms \eqref{eq10} for $rM$ and $M\oplus N$ obtained by direct sums from those of $M$ and $N$.
\end{prop}

\begin{proof} Each case is a trivial verification.
\end{proof}

\subsection{Weight gradings}

\begin{defn}\label{d5.1} A \emph{weight grading} on an additive $\otimes$-category $\sC$ is a family of endofunctors $w_n:\sC\to \sC$, for $n\in \Z$, 
provided with a natural isomorphism $\bigoplus_{n\in Z} w_n\iso Id_\sA$ and such that
\begin{enumerate}
\item $\sA(w_i(C),w_j(D))=0$ if $i\ne j$, for any $C,D\in \sC$;
\item If $C$ is of weight $i$ and $D$ is of weight $j$, then $C\otimes D$ is of weight $i+j$.
\end{enumerate}
(An object $C\in \sC$ is \emph{of weight $i$} if $w_j(C)=0$ for $j\ne i$. We then write $i=: w(C)$.)
\end{defn}

\begin{lemma}\label{l8} a) Let $\sC$ have a weight grading. If $C\in \sC$ is dualisable and of weight $i$, then its dual $C^\vee$ has weight $-i$.\\
b) Assume that $\sA$ has a weight grading for which $w(M)=1$. Then $w(L)=2$ and 
\[\omega(M^{\otimes a} \otimes (M^\vee)^{\otimes b}\otimes L^{\otimes c})^H=0 \]
if $a-b+2c\ne 0$ (here, $a,b\ge 0$ and $c\in \Z$). If $a-b+2c=0$, then
\[M^{\otimes a} \otimes (M^\vee)^{\otimes b}\otimes L^{\otimes c}\simeq M^{\otimes a+c} \otimes (M^\vee)^{\otimes a+c}.
\]
\end{lemma}

\begin{proof} a) Writing $C^\vee=\bigoplus w_j(C^\vee)$, (1) and (2) imply that the unit morphism $\un\by{\eta} C\otimes C^\vee$ factors through $C\otimes w_{-i}(C^\vee)$, and similarly for  the counit morphism $\epsilon$. But the identity morphism of $C^\vee$ equals
\[C^\vee\by{1_{C^\vee}\otimes \eta} C^\vee \otimes C\otimes C^\vee\by{\epsilon\otimes 1_{C^\vee}} C^\vee\]
hence factors through $ w_{-i}(C^\vee) \otimes C\otimes w_{-i}(C^\vee)$; reapplying (1) and (2), we get $w_j(C^\vee)=0$ for $j\ne -i$, as desired.

b) Since $L$ is invertible, it is irreducible hence has a weight. The first assertion then follows from a) and \eqref{eq10} (or from the proof of Lemma \ref{l6.1}), and the second follows from Lemma \ref{l7.2}. For the third, we distinguish two cases according as $c\ge 0$ or $c\le 0$. Note that, in any case, $a+c=b-c$ so this number is always $\ge 0$. If $c\ge 0$, we write
\begin{multline*}
M^{\otimes a} \otimes (M^\vee)^{\otimes b}\otimes L^{\otimes c}\simeq M^{\otimes a} \otimes (M^\vee)^{\otimes a}\otimes (M^\vee)^{\otimes 2c}\otimes L^{\otimes c}\\
\simeq M^{\otimes a} \otimes (M^\vee)^{\otimes a}\otimes (M^\vee)^{\otimes c} \otimes (M^\vee)^{\otimes c}\otimes L^{\otimes c}\\
\simeq M^{\otimes a} \otimes (M^\vee)^{\otimes a}\otimes (M^\vee)^{\otimes c} \otimes M^{\otimes c}\simeq M^{\otimes a+c} \otimes (M^\vee)^{\otimes a+c}.
\end{multline*}

The case $c\le 0$ is similar.
\end{proof}

\begin{prop}\label{p6.2} Assume that $M$ is semi-simple and that $\sA$ has a weight grading for which $w(M)=1$. Suppose moreover that, for any $n,r>0$, the composite homomorphism 
\[\sA(L,\Lambda^2(rM))^{\otimes n}\to \sA(L^n,\Lambda^2((rM))^{\otimes n})\to \sA(L^n,\Lambda^{2n}(rM))\]
induced by $\Lambda^2(rM)^{\otimes n}\to \Lambda^{2n}(rM)$ is surjective. Then $H_M= \bU(M)$.
\end{prop}

\begin{proof} As in the proof of Lemma \ref{l7.2}, we reduce to $K=Q$. Let $n,r>0$. 
By Proposition  \ref{p5.1}, $\bU(M)$ fixes $\End_\sA(rM)\simeq \sA(L,(rM)^{\otimes 2})$, hence also its direct summand $\sA(L,\Lambda^2(rM))$, hence also $\sA(L,\Lambda^2(rM))^{\otimes n}$, and therefore $\sA(L^n,\Lambda^{2n}(rM))$ by hypothesis; since $M^{\otimes 2n}$ is a direct summand of $\Lambda^{2n}(2nM)$, $\bU(M)$ fixes $\sA(L^n,M^{\otimes 2n})\simeq \End_\sA(M^{\otimes n})$ for all $n\ge 0$. By Lemma \ref{l7.2}, we have
\[\End_\sA(M^{\otimes n}) = \omega(M^{\otimes n}\otimes {M^\vee}^{\otimes n})^{H_M}.\]

By Lemma \ref{l8}, this shows that $\omega(T)^{H_M} = \omega(T)^{\bU(M)}$ for any $T\in \langle M\rangle$ of the form $M^{\otimes a} \otimes (M^\vee)^{\otimes b}\otimes L^{\otimes c}$. Since $M$ is semi-simple, $H_M$ is reductive and we get the conclusion by applying \cite[Prop. 3.1 (c)]{deligne-900} with $(G,H)=(\GL(\omega(M)),H_M)$.
\end{proof}

\begin{rk}\label{r6.1} The converse to Proposition \ref{p6.2} is true. Indeed, the map
\[(\omega(L^{-1}\otimes \Lambda^{2}(rM))^{\bU(M)})^{\otimes n}\to \omega(L^{-n}\otimes \Lambda^{2n}(rM))^{\bU(M)}\]
is surjective for all $r,n$ by \cite[Prop. 3.4]{milne2} (whose proof uses invariant theory). The claim now follows from Lemma \ref{l7.2}.
\end{rk}

\subsection{An ``easy'' exactness criterion}

Let 
\begin{equation}\label{eq7.1}
H\by{i} G\by{p} \Pi
\end{equation}
be a sequence of affine groups over a field $K$, where $i$ is a monomorphism and $pi=1$, i.e. $H\subseteq \Ker p$. 

\begin{prop}\label{p7.1} Assume $G, H$ proreductive (not necessarily connected). Then the following conditions are equivalent:
\begin{thlist}
\item $H=\Ker p$.
\item For any simple $S\in \Rep_K(G)$,  $S^H\ne 0$ $\Rightarrow$ $S^{\Ker p}\ne 0$.
\item For any $V\in \Rep_K(G)$,  $V^H=V^{\Ker p}$.
\end{thlist}
\end{prop}

(In other words, Condition H0 is sufficient in \cite[Lemma C.1]{TJM} when $G$ and $H$ are proreductive.)

\begin{proof}  (i) $\Rightarrow$ (ii) is obvious, and (ii) $\Rightarrow$ (iii) by semi-simplicity. For (iii) $\Rightarrow$ (i) we may restrict to the Tannakian subcategory of $\Rep_K(G)$ generated by one representation, hence assume $G$ of finite type. Then the conclusion follows from  \cite[Prop. 3.1 (c)]{deligne-900} just as in the proof of Proposition \ref{p6.2}.
\end{proof}

\subsection{Action of a profinite group}\label{s5.1} See also \cite{stacks} for a different presentation.

We keep the setting of \S \ref{s7.1}, and add a $\otimes$-action of a profinite group $\Pi$ on $\sA$. Namely, we are given a homomorphism $g\mapsto g_*$ from $\Pi$ to the monoid of strict $\Q$-linear $\otimes$-endofunctors $F$ of $\sA$ (strict means that $F(M)\otimes F(N)=F(M\otimes N)$ for any $M,N\in \sA$), and $g_*\un=\un$ for all $g\in \Pi$.

\begin{defn}\label{d7.1}
Given an open subgroup $U$ of $\Pi$, we say that an object $M\in \sA$ is \emph{$U$-centered} if $g_*M=M$ for all $g\in U$; we say that $M$ is \emph{centered} if it is $U$-centered for some $U$. If $M$ is $U$-centered, $U$ acts on $\End_\sA(M)$. We say that $\Pi$ acts \emph{continuously} on $\sA$ if
\begin{itemize}
\item any object $M\in \sA$ is isomorphic to a centered object;
\item if $M$ is $U$-centered, the action of $U$ on $\End_\sA(M)$ is continuous (i.e., the stabiliser of any endomorphism is open).
\end{itemize}
\end{defn}

\begin{lemma}\label{l7.5} Suppose that $\Pi$ acts continuously on $\sA$.\\
a) If $M$ and $N$ are $U$-centered, so are $M\oplus N$, $M\otimes N$ and $M^\vee$.\\
b) If $M$ and $N$ are $U$-centered, then $U$ acts continuously on $\sA(M,N)$.\\
c) If $M$ is centered, then any direct summand of $M$ is centered.
\end{lemma}

\begin{proof} a) is obvious. b) follows from a) because $\sA(M,N)$ is a direct summand of $\End_\sA(M\oplus N)$. For c), let $e\in \End_\sA(M)$ be the idempotent corresponding to a direct summand $N$ of $M$. If $M$ is $U$-centered, then by hypothesis there is an open subgroup $V\subseteq U$ such that $g(e)=e$ for any $g\in V$; equivalently, $N$ is $V$-centered.
\end{proof}

We now assume that the action of $\Pi$ is continuous. By Lemma \ref{l7.5}, the full subcategory of $\sA$ consisting of centered objects is equivalent to $\sA$ and Tannakian (even though it may not be closed under extensions, it is abelian since it is equivalent to an abelian category). Without loss of generality, \emph{we henceforth assume that every object of $\sA$ is centered.}

Let $\sB_1$ be the category of descent data of $\sA$ with respect to the action of $\Pi$: an object of $\sB_1$ is a system $(M,u_g)_{g\in \Pi}$ where $M\in \sA$ and $u_g:M\iso g_*M$ are isomorphisms such that $u_{gh} = g_*u_h \circ  u_g$ for any $g,h\in \Pi$, and a morphism from $N=(M,u_g)$ to $N'=(M',u'_g)$ is a morphism from $M$ to $M'$ which commutes with $u_g,u'_g$ in an obvious sense. We then get an action of $\Pi$ on $\sA(M,M')$ be the formula
\begin{equation}\label{eq7.5}
g(f) = {u'_g}^{-1}g_*f u_g, \quad g\in \Pi,\quad f\in \sA(M,M')
\end{equation}
so that
\begin{equation}\label{eq7.4}
\sB_1(N,N')=\sA(M,M')^\Pi.
\end{equation}

Then $\sB_1$ inherits a $Q$-linear $\otimes$-structure by
\[(M,u_g)\otimes (N,v_g) = (M\otimes N,u_g\otimes v_g) \]
for which it is rigid, and a $\otimes$-functor $\iota_1:\sB_1\to \sA$, $(M,u_g)\mapsto M$.

Let $(M,u_g)\in \sB_1$, and let $U\subseteq \Pi$ be such that $M$ is $U$-centered. Then the $u_g$ ($g\in U$) define an action of $U^\op$ on $M$.

\begin{defn}\label{d7.2} The descent datum $(M,u_g)$ is \emph{continuous} if $u_g=1_M\, \forall g\in V$ for a suitable open subgroup $V\subseteq U$. We write $\sB$ for the full subcategory of $\sB_1$ consisting of continuous descent data, and $\iota:\sB\inj \sA$ for the restriction of $\iota_1$ to $\sB$.
\end{defn}

\begin{prop}\label{p7.2} a) If $(M,u_g),(M',u'_g)\in \sB$, the action \eqref{eq7.5} is continuous.\\
b) The category $\sB$ is abelian; the functor $\iota$ is faithful and exact.\\
c) Any object $M$ of $\sA$ is a direct summand of an object of the form $\iota(N)$ for $N\in \sB$; if $G=\Aut^\otimes(\omega\iota)$, the natural homomorphism $\iota^*:H\to G$ is a monomorphism.\\
d) Let $\sB^0$ be the full subcategory of $\sB$ formed of the objects $(M,u_g)$ where $M$ is of the form $\un^r$ for some $r\ge 0$, and let $\lambda:\sB^0\inj \sB$ be the full embedding. Then $\sB^0$ is a rigid $\otimes$-subcategory of $\sB$ and the functor $\theta:\sB^0\to \Rep_Q(\Pi)$ defined by the action of $\Pi$ on the $\iota M$'s is an equivalence of $\otimes$-categories.\\
e) Given $N\in \sB^0$, any subobject $P$ of $\lambda(N)$ is isomorphic to an object of the form $\lambda(N')$ for $N'$ a subobject of $N$. The homomorphism $\pi:G\to \Pi$ induced by d) is faithfully flat.
\end{prop}

\begin{proof} a) and b) are obvious. For c), let $M'=\bigoplus_{g\in \Pi/U} g_*M$ (a finite sum). Then $M'$ is provided with a canonical continuous descent datum $(u_g:M'\iso g_*M')$, which permutes the summands ($u_g=1_{M'}$ if $g\in U$). We take $N=(M',u_g)$.  The last claim then follows from  \cite[Prop. 2.21 (b)]{dm}. 

d) If $V\in \Rep_Q(\Pi)$ has dimension $n$, the choice of a basis of $V$ provides an isomorphism $V\simeq \theta(\un^n,u_g)$ where $u_g$ is defined by the action of $g\in\Pi$ on $V$, so $\theta$ is essentially surjective. It is also fully faithful by definition of the morphisms in $\sB$.

e) Let $N=(M,u_g)$. Since $\un$ is simple \cite[Prop. 1.17]{dm}, $M\simeq \un^n$ is semi-simple. If $P=(N',u'_g)$, then $N'\simeq \un^m$ for some $m\le n$. The last claim then follows from \cite[Prop. 2.21 (a)]{dm}.
\end{proof}

In Proposition \ref{p7.2}, I don't know if the sequence
\begin{equation}\label{eq7.2}
1\to H\by{\iota^*} G\by{p}\Pi\to 1
\end{equation}
is exact at $G$ in general\footnote{See however \cite[Theorem 4.6]{stacks}.}. This is true if $\sA$ is semi-simple:

\begin{prop} \label{l7.3} If $\sA$ is semi-simple,\\
a)  so is $\sB$.\\
b) Let $S\in \sB$ be simple. Then $S\in \sB^0$ if and only if $\sA(\un,\iota S)\ne 0$; we then have $S\simeq \sA(\un,\iota S)$, where the right hand side is viewed as a $\Pi$-module by \eqref{eq7.5}.\\
c) The full embedding $\lambda:\sB^0\to \sB$ has the (non mo\-no\-id\-al) right adjoint $\rho:M\mapsto  \sA(\un,\iota M)$.\\
d) The sequence \eqref{eq7.2} is exact.
\end{prop}

\begin{proof} a) Let $N\in \sB$. By \eqref{eq7.4}, $\End_\sB(N)$ is the centraliser of the semi-simple algebra  $Q\bar\Pi$ in $\End_\sA(\iota N)$, where $\bar \Pi$ (a finite group) is the image of $\Pi$ in $\mathrm{Aut}_\sA(\iota N)$ (cf. Proposition \ref{p7.2} a)). Since the latter is semi-simple, so is $\End_\sB(N)$  \cite[\S 14, no 5, th. 5 a)]{bbki}.

b) Write $\iota S\simeq \bigoplus_\alpha S_\alpha^{n_\alpha}$ where the  $S_\alpha$'s are simple and pairwise non-isomorphic. If $\sA(\un,\iota S)\ne 0$, then $S_{\alpha_0}\simeq \un$ for some $\alpha_0$, and $\sA(\un,S_\alpha)=\sA(S_\alpha,\un)=0$ for any $\alpha\ne \alpha_0$. But then, $\sA(\un,g_*S_\alpha)=\sA(g_*S_\alpha,\un)=0$ for any $\alpha\ne \alpha_0$. Thus  $\un^{n_{\alpha_0}}$ splits off as a direct summand of the descent datum $S$, and we must have $\iota S\simeq \un^{n_{\alpha_0}}$; the isomorphism $S\simeq \sA(\un,\iota S)$ is then clear.  The converse is obvious. 

In c), $\rho$ exists by \cite[Prop. 5.3]{adjoints}, which also gives the following recipe: for $S$ simple, $\rho(S)=S$ if $S\in \sA^0$ and $\rho(S)=0$ otherwise. It suffices to check that this matches with the formula of the statement, which follows from b).

In d), the exactness at $H$ (resp. $\Pi$) was proven in Proposition \ref{p7.2} c) (resp. e)). For the exactness at $G$,  we are reduced by a)  to checking Condition (ii) of Proposition \ref{p7.1}. By Lemma \ref{l7.2}, we have the isomorphism $\omega(\iota S)^H=\sA(\un,\iota S)\otimes K$. If it is nonzero, $S\in \sB^0$ by b). But then the action of $G$ on $\omega(\iota S)$ factors through $\Pi$, hence $\Ker p$ acts trivially on $\omega(\iota S)$.
\end{proof}

We now assume that $\sA$ is semi-simple as in Proposition \ref{l7.3}, and has a weight grading. Let $N\in \sB$ be such that $M=\iota N$ is weakly polarisable with respect to $L\in \sA$ as in \S \ref{s7.1}, and verifies the hypothesis of Proposition \ref{p6.2}. Applying the latter and Proposition \ref{l7.3} to $\langle N\rangle \inj \langle M\rangle$, we get a short exact sequence
\begin{equation}\label{eq7.3}
1\to \bU(M)\by{\iota^*} G_M\by{\pi}\Pi_M\to 1
\end{equation}
where $G_M$ (resp. $\Pi_M$) is the Tannakian group of $\langle M\rangle$ (resp. of $\langle M\rangle^0:=\langle M\rangle\cap \sB^0$. It remains to compute $\Pi_M$.

\begin{lemma}\label{l7.4} The $\otimes$-category $\langle N\rangle^0$ is generated by $\End_\sA(M)$, viewed as a $\Pi$-module. Consequently, the group $\Pi_M$ in \eqref{eq7.3} is the smallest quotient of $\Pi$ which acts nontrivially on $\End_\sA(M)$.
\end{lemma}

\begin{proof} By Proposition \ref{l7.3} c) applied to the inclusion $\langle N\rangle \inj \langle M\rangle$, $\langle M\rangle^0$ is generated by the objects $\sA(\un,\iota N')$ for $N'\in \langle N\rangle$. This is $0$ unless $\iota N'$ has weight $0$, so, reasoning as in the proof of Proposition \ref{p6.2}, we may restrict to $\sA(L^n,\iota N')$ for $\iota N'$ of the form $\Lambda^{2n}(rM)$ ($n,r>0$), and then to $n=1$ by the hypothesis of Proposition \ref{p6.2}. But $\sA(L,\Lambda^2(rM))$ is a direct sum of copies of $\sA(L,\Lambda^2(M))$ and of $\sA(L,M^{\otimes 2})\simeq \End_\sA(M)$, and the former is a direct summand of the latter. 
\end{proof}

\subsection{Proof of Theorem \ref{t5}}\label{s7.2} We first consider $\sA=\LMot_\num(k)_0$ as in Corollary \ref{c1} (here $Q=\Q$); the hypotheses on semi-simplicity and weight grading are granted by this corollary. We take $M=Lh^1(A)$: note that it generates the same $\otimes$-subcategory as $Lh(A)$, by the same corollary.  In view of \eqref{eq5.1}, any polarisation of $A$ yields an isomorphism  \eqref{eq10}. Moreover,  Lemma \ref{l5.1} implies that the surjectivity hypothesis of Proposition \ref{p6.2} is verified, because it implies that the morphism  $L^1(A^r)^{\otimes n}\to L^n(A^r)$ in this proposition is the one given by the intersection product: this computes the Tannakian group of $\langle A\rangle$ as $\bU_A$ (cf. \eqref{eq0.4}). By Proposition \ref{p5.1}, this yields an exact sequence
\[1 \to \prod_{A\in S_0} \SU_A\to \Aut^\otimes(\omega)\to \G_m\to 1\]
where $S_0$ is the set of isomorphism classes of simple abelian $k$-varieties. 

When $k$ is separably closed, we have $\LMot_\num(k)=\LMot_\num(k)_0$: this proves the isomorphism of Theorem \ref{t5} a). 

To prove b), we apply Proposition \ref{l7.3} and Lemma \ref{l7.4}. For this, we need first to show that the action of $\Gamma$ on  $\LMot(k_s)$ described after Definition \ref{d1} is continuous in the sense of Definition \ref{d7.1}, and then that $\LMot_\num(k)$ is equivalent to the corresponding category of continuous descent data in the sense of Definition \ref{d7.2}.

Let $A\in \Ab(k_s)$. Then $A$ is defined over some finite subextension $E/k$ of $k_s/k$, which means that there exists $A_0\in \Ab(E)$ and an isomorphism $A\simeq A_0\otimes_E k_s$. Hence, if $U=Gal(k_s/E)$ and $n\in \Z$, $Lh(A)\otimes \L^n$ is isomorphic to the $U$-centered motive $M\otimes_E k_s$ where $M=Lh(A_0)\otimes \L^n\in \LMot_\num(E)$. Moreover, $R=\End_{\LMot_\num(k_s)}(M)$ is the colimit of the $\End_{\LMot_\num(E')}(M\otimes_E E')$ for $E'\subseteq k_s$ a finite extension of $E$, hence $U$ acts continuously on $R$. Therefore the action of $\Gamma$ on $\LMot(k_s)$ is continuous, thanks to Lemma \ref{l7.5} c).

The canonical continuous descent datum on $A_{k_s}$ attached to $A\in \LCorr_\num(k)$  provides a $\otimes$-functor from $\LCorr_\num(k)$ to the category of continuous descent data on $\LCorr_\num(k_s)$ for the action of $\Gamma$. This functor is fully faithful by \eqref{eq4.1} and \eqref{eq7.4},   hence extends to a fully faithful functor from $\LMot_\num(k)$ to the category $\sB$ of continuous descent data on $\sA=\LMot_\num(k_s)$. 

It remains to show its essential surjectivity. Let $N=(M,u_g)\in \sB$. Choose a finite subextension  $E/k$ of $k_s/k$ and $M_0\in \LMot_\num(E)$ such that $M\simeq (M_0)_{k_s}$; by the continuity of $N$, up to enlarging $E$ we may further assume that $U=Gal(k_s/E)$ is such that $u_g=1$ for $g\in U$, and moreover is normal in $\Gamma$. Let $G=\Gamma/U=Gal(E/k)$; for $g\in G$, $u_g$ descends to an isomorphism $M_0\iso g_* M_0$ that we still denote by $u_g$. Moreover we have a canonical isomorphism $\lambda_E(M_0)\simeq \lambda_E(g_* M_0)$, where $\lambda_E$ is the right adjoint of Lemma \ref{l4.1}. This implies that the $\lambda_E(u_g)$'s define a homomorphism $\rho$ from $\Q[G]$ to the endomorphism ring of $\lambda_E(M_0)$. Let $e=\frac{1}{|G]} \sum_{g\in G} \rho(g)$ be the corresponding idempotent, and let $M_1=\IM e\in \LMot_\num(k)$. I claim that $M_1\mapsto N$ under $\LMot_\num\to \sB$. Indeed, the inclusion $M_1\to \lambda_E(M_0)$ yields by adjunction a morphism $i_E M_1 \to M_0$, which is seen to be an isomorphism by using \eqref{eq5.6}. This computation also shows that the canonical descent datum of $i_EM_1$ matches with that of $N$.

This concludes the proof of Theorem \ref{t5}. \qed

\begin{rk}  Proposition \ref{l7.3} also applies to prove \cite[Prop. 6.23 (a)]{dm} -- except for the connectedness of $G(k_s)$ which would require the ``Hodge = absolute Hodge'' conjecture --, and \cite[4.6, exemples]{pour}. It also applies to numerical Grothendieck motives under the standard conjecture D or, unconditionally, to this category restricted to numerical motives of abelian type in characteristic $0$. Since the Tannakian group $G_s$ is not connected (see \cite[Table 2 p. 655]{milne2}), one sees that the expected connectedness of the motivic Galois group over separably closed fields cannot be hoped to be proven by purely formal arguments in the above style.
\end{rk}

\section{Remarks and questions}

\begin{rks} 
1) The algebra $\sB_{n,\Q}$ defined by Ancona in \cite[Def. 5.2]{ancona} is contained in $\End_{\LMot(k)}(h^1(A)^{\otimes n})$: this allows one to refine his lifting results. Similarly, O'Sullivan's lifting theorem of \cite{os} refines to $\LMot(k)$.\\
2)  For an abelian $k$-variety $A$, let $\tilde L^*(A)=L^*(A_{k_s})^G
$; equivalently, this is the subalgebra of $CH^*(A)$ generated by transfers of intersections of divisor classes over finite separable extensions of $k$. One would like to compute $\tilde L^*(A)$ in the style of Theorem \ref{p1}.  We have the following partial result:
\end{rks}

\begin{prop} Suppose that $G$ acts trivially on $CH^1_0(A_{k_s})\simeq \NS(A_{k_s})$. Then $\tilde L^*(A)$ is generated by $CH_1^0(A)$ and $CH_0(A)$ under the Pontrjagin product.
\end{prop}

\begin{proof} By the isomorphism in the proof of Lemma \ref{l0}, $G$ acts trivially on $CH_1^0(A_{k_s})$. It follows by a transfer argument that the map $CH_1^0(A)\to CH_1^0(A_l)$ is surjective.  The conclusion then follows from Theorem \ref{p1} and Lemma \ref{l2} b). (A variant would be to use \cite[Prop. 4 a)]{beauville2}.)
\end{proof}

\begin{qn} More varieties than abelian varieties have a motive in $\LMot(k)$. For example a (geometrically connected) curve $C$, since $h^1(C)\simeq h^1(J(C))$ and the Chow-Künneth projectors of $J(C)$ are in $\LMot(k)$ by Corollary \ref{c1}. Therefore, products of curves as well. In general, let $X$ be a smooth projective variety. Suppose that the motive of $X$ in $\sM$ is a direct summand of the motive of an abelian variety $A$. When is the corresponding projector in $L^*(A\times A)$? Special case: there exists a dominant rational map $A\tto X$.

Variant: in Corollary \ref{c1}, take for $\sim$ numerical equivalence. Then \eqref{eqb0} is a functor between semi-simple abelian categories. Let $S$ be a simple object of $\LMot_\sim(k)$. Does it remain simple in $\sM_\sim(k)$?

In particular, do the refined (Chow-)Künneth decompositions of \cite[Th. 7.7.3]{kmp}
\[h^i(A) \simeq \bigoplus_{j} h^{i,j}(A)\otimes \L^j\]
coincide in the two categories? This is true for $i=2$ by a generalisation of \cite[Proof of Prop. 7.2.3]{kmp}.
However the answer is no in general, as pointed out by Peter O'Sullivan by the following simple argument:

\emph{Suppose for example that $k$ is separably closed and we work modulo homological = numerical equivalence (Theorem \ref{t3} a)). There is then an abelian variety $A$ such that for some $i$ the dimension $d'$of $L^i_\num(A)$ over $\Q$ is strictly less than the dimension $d$ of $A^i_\num(A)$ over $\Q$. Then $Lh^{2i,i}(A)$ is the direct sum of $d'$ copies of the unit motive, while $h^{2i,i}(A)$ is the direct sum of $d > d'$ copies. It follows that $Lh^{2i,j}(A)$ for some  $j < i$ has a simple direct summand which is not simple in the category of ordinary motives.}
\end{qn}

\end{document}